\documentclass[oneside]{amsart}
\usepackage[foot]{amsaddr}
\usepackage{amsrefs}
\usepackage{amssymb}
\usepackage{enumitem}
\usepackage{mathtools}
\usepackage{varioref}
\usepackage{pifont}
\usepackage{tikz}
\usetikzlibrary{calc}
\usetikzlibrary{arrows}
\usepackage[cmtip,all]{xy}

\newtheorem{Theorem}{Theorem}

\newtheorem{Lemma}{Lemma}[section]
\newtheorem{Corollary}[Lemma]{Corollary}
\newtheorem{Proposition}[Lemma]{Proposition}

\numberwithin{equation}{section}

\newcommand\QQ{{\mathbb{Q}}}
\newcommand\ZZ{{\mathbb{Z}}}
\newcommand\NN{{\mathbb{N}}}
\newcommand\A{\mathcal{A}}
\DeclareMathOperator\Aut{Aut}
\DeclareMathOperator\Frac{Frac}
\DeclareMathOperator\Spec{Spec}
\DeclareMathOperator\SL{SL}
\DeclareMathOperator\GL{GL}
\newcommand\wt[1]{{(#1)}}
\newcommand\der[1]{\tfrac{\partial}{\partial#1}}

\DeclarePairedDelimiter\lin{\langle}{\rangle}
\DeclarePairedDelimiter\abs{\lvert}{\rvert}

\renewcommand\epsilon{\varepsilon}
\newcommand\psm[1]{\begin{psmallmatrix}#1\end{psmallmatrix}}
\newcommand\sym{\mathrm{sym}}

\newlist{thmlist}{enumerate}{1}
\setlist[thmlist]{
  nolistsep,
  ref={\mdseries\emph{(\roman*)}},
  label=\emph{(\roman*)}, 
  }
\newcommand\thmitem[1]{\textup{(\textit{\romannumeral #1})}}

\newlength\circwidth
\newcommand\circitem[1]{\textup{(%
  \makebox[\circwidth]{%
    \ifcase#1%
    \or\ding{168}%
    \or\ding{171}%
    \or\ding{169}%
    \or\ding{171}%
    \fi
  }%
)}}

\begin{document}

\title[Automorphisms and isomorphism]{Automorphisms and isomorphism of
quantum generalized Weyl algebras}

\author{Mariano Su\'arez-Alvarez}
\email{mariano@dm.uba.ar}

\author{Quimey Vivas}
\address{Departamento de Matemática.
Facultad de Ciencias Exactas y Naturales.
Universidad de Buenos Aires.
Ciudad Universitaria, Pabellón I.
(1428) Ciudad de Buenos Aires (Argentina)}
\email{qvivas@dm.uba.ar}


\date{December 26, 2012}


\begin{abstract}
We classify up to isomorphism the quantum generalized Weyl algebras and
determine their automorphism groups in all cases in a uniform way,
including where the parameter~$q$ is a root of unity, thereby completing
the results obtained by [Bavula, V. V.; Jordan, D. A. Isomorphism problems
and groups of automorphisms for generalized Weyl algebras. Trans. Amer.
Math. Soc. 353 (2001), no. 2, 769--794] and [Richard, L.; Solotar, A.
Isomorphisms between quantum generalized Weyl algebras. J. Algebra Appl. 5
(2006), no. 3, 271--285]
\end{abstract}

\maketitle

\section*{Introduction}

If $k$ is a field, $q\in k\setminus\{0,1\}$, $D$ is one of $k[h]$
or $k[h^{\pm1}]$ and $a\in D\setminus0$, the \emph{quantum generalized
Weyl algebra} $A=\A(D,q,a)$ is the $k$-algebra freely generated by
letters $y$,~$x$, $h$ (and its inverse  $h^{-1}$ when $D=k[h^{\pm1}]$)
subject to the relations 
  \begin{align*}
  & hy = qyh,    && xh = qhx,    && yx = a(h),   && xy = a(qh).
  \end{align*}
This construction, a special case of a general one introduced by
V.V.\,Bavula in~\cite{B}, provides an interesting class of algebras
containing the quantum plane, the quantum Weyl algebra, certain
well-known quotients of the quantum enveloping algebra $\mathcal
U_q(\mathfrak{sl}_2)$ related to the primitive quotients of the
classical enveloping algebra~$\mathcal U(\mathfrak{sl}_2)$, studied by
J.\,Alev and F.\,Dumas in~\cite{AD:2}, some invariant subalgebras of
these under finite group actions, the so-called ambiskew polynomial
rings, and several other examples. They have notably appeared also
under the name of non-commutative deformations of Kleinian
singularities of type $A$ in work of T.J.\,Hodges~\cite{H} and are,
in~fact, somewhat ubiquitous.

It is the purpose of this paper to present a solution to the problem
---initially posed by Hodges in~\cite{H} in general--- of determining which
pairs of quantum generalized Weyl algebras are isomorphic, and to
describe the automorphism groups of these algebras. Our first result
solves the isomorphism problem:

\begin{Theorem}\label{thm:isos}
The two algebras $A_1=\A(D,q_1,a_1)$ and $A_2=\A(D,q_2,a_2)$, with
$a_1$ and $a_2$ non-units, are isomorphic if and only if
$q_2\in\{q_1,q_1^{-1}\}$ and there exist a unit $\alpha\in D$, a
non-zero scalar~$\beta\in k$ and $\epsilon\in\{\pm 1\}$ such that
$a_2(h)=\alpha a_1(\beta h^\epsilon)$. If~$D=k[h]$ then necessarily
$\epsilon=1$ and $\alpha \in k^\times$.
\end{Theorem}
 
Let us single out two interesting special cases of this theorem. If
$q\in k\setminus\{0,1\}$, the \emph{quantum  plane} is the algebra
$M_q=k\lin{x,y:yx=qxy}$ and the \emph{quantum Weyl algebra} is the
algebra $A^1_q(k)=k\lin{x,y:yx-qxy=1}$. As $M_q\cong\A(k[h],q,h)$ and
$A_q^1(k)\cong\A(k[h],q,h-1)$, Theorem~\ref{thm:isos} tells us that
for all $q_1$,~$q_2\in k\setminus\{0,1\}$ we have $M_{q_1}\cong
M_{q_2}$ iff $A_{q_1}^1(k)\cong A_{q_2}^2(k)$ iff
$q_2\in\{q_1,q_1^{-1}\}$. This characterization of isomorphisms
between quantum planes and between quantum Weyl algebras had been
established when the parameters are \emph{not} roots of unity by
Alev and Dumas in~\cite{AD}.

\begin{Theorem}\label{thm:autos-1}
Let $A=\A(k[h],q,a)$ be a quantum generalized Weyl algebra with $a$
not a unit, let $N=\deg a$ and write $a=\sum_{i=0}^Na_ih^i$, and let
  \(
  g=\gcd\{i-j:a_ia_j\neq 0\}
  \)
and $C_g\subseteq k^\times$ be the subgroup of $g$th roots of unity;
if~$a$~is a monomial, we make the convention that~$g=0$ and
$C_g=k^\times$. If $(\gamma,\mu)\in C_g\times k^\times$, there is an
automorphism $\eta_{\gamma,\mu}:A\to A$ such that
$\eta_{\gamma,\mu}(y)=\mu y$, $\eta_{\gamma,\mu}(h)=\gamma h$ and
$\eta_{\gamma,\mu}(x)=\mu^{-1}\gamma^N x$. The~set
$G=\{\eta_{\gamma,\mu}:(\gamma,\mu)\in C_g\times k^\times\}$ is a
subgroup of~$\Aut(A)$ isomorphic to~$C_g\times k^\times$. 
\begin{thmlist}

\item If $q\neq-1$, we in fact have $\Aut(A)=G$, and

\item if $q=-1$, there is a right split short exact sequence of groups
  \[
  \xymatrix{
  1 \ar[r]
    & G \ar@{^(->}[r]
    & \Aut(A) \ar[r]
    & \ZZ/2\ZZ \ar[r]
    & 1
  }\]
The cyclic group~$\ZZ/2\ZZ$ appearing here is generated by the image
of the involutory automorphism $\Omega:A\to A$ such that
$\Omega(y)=x$, $\Omega(h)=-h$ and~$\Omega(x)=y$.

\end{thmlist}
\end{Theorem}

\noindent To state the analog of this theorem for the case where
$D=k[h^{\pm1}]$, we need a definition. We say that a Laurent
polynomial $f\in k[h^{\pm 1}]$ is \textit{symmetric} if there exist
$l\in\NN$, $\gamma\in k$ and $\delta\in k$ such that
$\delta f(h)=h^l  f(\gamma h^{-1})$.

\begin{Theorem}\label{thm:autos-2}
Let $A=\A(k[h^{\pm1}],q,a)$ be a quantum generalized Weyl algebra, with
$a=\sum_{i\in I}a_ih^i$ a non-unit in $k[h^{\pm1}]$, and let
  \(
  g=\gcd\{i-j:a_ia_j\neq 0\}
  \)
and $C_g\subseteq k^\times$ be the subgroup of $g$th roots of unity;
fix $i_0\in I$. If $(\gamma,\mu)\in C_g\times k^\times$, there is an
automorphism $\eta_{\gamma,\mu}:A\to A$ such that
$\eta_{\gamma,\mu}(y)=\mu y$, $\eta_{\gamma,\mu}(h)=\gamma h$ and
$\eta_{\gamma,\mu}(x)=\mu^{-1}\gamma^{i_0} x$. The set
$G=\{\eta_{\gamma,\mu}:(\gamma,\mu)\in C_g\times k^\times\}$ is a
subgroup of~$\Aut(A)$ isomorphic to~$C_g\times k^\times$. Consider the
subgroup $K$ of $\Aut A$ of all automorphisms $\eta$ such that
$\eta(h)$ is a scalar multiple of $h$.
\begin{thmlist}

\item If $a$ is symmetric then $\Aut A\cong K \ltimes \ZZ/2\ZZ$ and,
if not, $\Aut(A)=K$.

\item If $q=-1$ then $K\cong G\ltimes \ZZ/2\ZZ$ and otherwise $K=G$.

\end{thmlist}
\end{Theorem}

\noindent To avoid complicating statements and proofs, we have chosen
to postpone to the end of the paper the results in the line of these
three theorems for the case in which the polynomial $a$ is invertible
in the ring $D$.

\medskip

The results corresponding to these theorems for the case of classical
generalized Weyl algebras ---in which ``there is no $q$''--- have been
given by Bavula and Jordan in~\cite{BJ} and the quantum case
as above but with~$q$ not a root of unity has been solved for
$D=k[h]$ by L.\,Richard and A.\,Solotar in~\cite{RS} and for
$D=k[h^{\pm1}]$ by Bavula and Jordan also in~\cite{BJ}.

Our approach makes no hypothesis on the scalar parameter, and it is
interesting to remark one key point which makes the difference.
In~\cite{AD}, Alev and Dumas attached to a $k$-algebra~$\Lambda$ the
subgroup~$G(\Lambda)=(\Lambda^\times)'\cap k^\times\subseteq k^\times$
---where $(\Lambda^\times)'$ is the derived subgroup of the group of
units of~$\Lambda$--- and showed that if $k_q(x,y)$ denotes the
quantum Weyl field we have $G(k_q(x,y))=\lin{q}$, the cyclic subgroup
generated by~$q$. Richard and Solotar prove that the fraction field of
a quantum generalized Weyl algebra $A=\A(q,a)$ is isomorphic
to~$k_q(x,y)$ and, since in their situation $q$ is not a root of
unity, notice that one can recover~$q$ from~$A$, up to inversion, as
one of the two generators of~$G(\Frac A)$. If instead $q$~has finite
order in~$k^\times$, the subgroup~$\lin{q}$ has many generators and
their approach cannot get started. We replace below their
consideration of~$G(\Frac A)$ by a detailed study of certain
derivations of~$A$ and their eigenvalues, and this avoids that
difficulty: in a very loose sense, this is like ``taking the
logarithm'' of $G(\Frac A)$. Similar difficulties with parameters of
finite order appear when trying to classify other classes of algebras,
like that of down-up algebras introduced by G.\,Benkart and T.\,Roby
in~\cite{BR}, and one can hope that similar ideas may possibly
overcome these too.

In~\cite{SSAV}, together with A. Solotar, we computed the Hochschild
cohomology of quantum generalized Weyl algebras defined over $k[h]$.
The results of the present paper arose in the process of studying the
algebraic structure of the cohomology ---the cup product and the
Gerstenhaber bracket.

\medskip

We finish by emphasizing that the theorems stated above, as well as
all the related work we referred to, exclude the case where~$q=1$,
which is precisely that in which the algebras are commutative. When
$D=k[h]$, the problem of determining the automorphisms is
that of finding the automorphism group of the affine surface $\Spec
k[x,y,h]/(xy-a(h))$. L.\,Makar-Limanov gave in~\cite{ML} explicit
generators for these groups and recently J.\,Blanc and A.\,Dubouloz
showed in~\cite{BD} that they have an amalgamated product structure
similar to that of~$\Aut(k[x,y])$ described by the classical theorems
of L.\,Makar-Limanov,  H.W.E.\,Jung and W.\,van der Kulk, and that the
surfaces are classified under isomorphism exactly as in
Theorem~\ref{thm:isos}. While Makar-Limanov deals systematically with
locally nilpotent derivations, as we do, the methods with which these
commutative results are obtained are quite different from ours ---the
work~\cite{BD}, for example, is a paper on algebraic geometry.

\section{Preliminaries}
\label{prel}

We fix a field~$k$ of characteristic zero and identify $\QQ$ with its
prime field. If $q\in k$, $D$ is one of $k[h]$ or $k[h^{\pm 1}]$ and
$a\in D$, the \emph{quantum generalized Weyl algebra} $A=\A(D,q,a)$ is
the $k$-algebra freely generated by letters $y$,~$x$,~$h$ (and
$h^{-1}$ when $D=k[h^{\pm 1}]$) subject to the relations; 
  \begin{align*}
  & hy = qyh,    && xh = qhx,    && yx = a(h),   && xy = a(qh).
  \end{align*}
The set $\{y^ih^jx^k:ik=0\}$ is a $k$-basis of~$A$; we call its
elements \emph{standard monomials}. The algebra is a domain iff
$q\neq0$ and $a\neq0$: we will always assume this is the case. We will
moreover suppose throughout that $q\neq1$, thereby excluding all the
commutative examples and no other.

We write $a=\sum_{i=M}^N a_i h^i$ with $a_Ma_N\neq 0$. Notice that if
$a$ is a unit, that is, if $M=N=0$ when $D=k[h]$ or $M=N$ when
$D=k[h^{\pm 1}]$, then $A$ is isomorphic to the Ore extension
$D[x^{\pm 1},\sigma]$. As the results and methods needed to deal with
this case are different, we will do this separately at the end of this
paper. 

The algebra~$A$ is $\ZZ$-graded in a unique way so that the degrees
of~$y$,~$h$ and~$x$ are $1$,~$0$, and~$-1$, respectively; we refer to
the degree~$\abs{a}$ of an homogeneous element $a\in A$ in this
grading as its \emph{weight}, and extend this convention to related
contexts. For $r\in\ZZ$, we let $A^\wt r$ be the homogeneous component
of~$A$ of degree~$r$; we have $A^\wt0=k[h]$ and, for each $r\in\NN$,
$A^\wt r=y^rk[h]$ and $A^\wt{-r}=k[h]x^r$. 

\medskip

Let $V=\bigoplus_{i\in \ZZ}V_i$ be a graded vector space and let
$d:V\to V$ a not necessarily homogeneous linear endomorphism. We say
that $d$ is \emph{locally finite} if for each $v\in V$ the cyclic
subspace $\lin{v}_d$ of $V$ generated by $d$ and $v$ is
finite-dimensional, and that $d$ is \textit{locally nilpotent} if for
each $v\in V$ we have $d^i(v)=0$ for $i\gg0$. It is enough to check
these conditions on homogeneous elements of~$V$.

\begin{Lemma}
Suppose $d=d_1+\dots+d_l$ with $d_1$,~\dots,~$d_l:V\to V$
\emph{homogeneous} endomorphisms of~$V$ of degrees
$\alpha_1$,~\dots,~$\alpha_l$ such that $\alpha_1<\dots<\alpha_l$.
If $d$ is locally-finite then $d_1$ and $d_l$ are locally-finite.~\qed
\end{Lemma}

An homogeneous endomorphism of~$V$ of non-zero degree is locally
finite iff it is locally nilpotent. It follows that if in the lemma we
have $\alpha_l\neq0$ then in fact $d_l$ is locally nilpotent, and
similarly for~$d_1$.

\medskip

Let now $A$ be a graded algebra. If $d:A\to A$ is a homogeneous
derivation of positive degree which is locally nilpotent, there is a
function $\deg_d:A\setminus0\to\NN$ such that for each $u\in
A\setminus0$ we have $\deg_d(u) = \max\{r\in \NN_0 : d^r(u)\neq 0\}$.
It is straightforward to check that $\deg_d$ is such that for all
$u$,~$v\in A$ we have
  \begin{gather*}
  \deg_d(u+v)\leq \max\{\deg_d(u),\deg_d(v)\}, \\
  \deg_d(uv) = \deg_d(u)+\deg_d(v).
 \end{gather*}
It follows from this that the subalgebra $\ker d$ is \emph{factorially
closed}: if $u$,~$v\in A\setminus0$ then
  \(
  d(uv)=0 \implies d(u)=d(v)=0
  \).
In particular, $d$ vanishes on the units of $A$.

In contexts where this makes sense, we will write $x\doteq y$ to mean
that $y$ is a non-zero scalar multiple of~$y$.

\section{Derivations}
\label{sect:ders}

Let $A=\A(D,q,a)$ be a quantum generalized Weyl algebra. If
$u_1$,~$u_2$,~$u_3\in A$, we write $u_1\der{y}+u_2\der{h}+u_3\der{x}$
the unique derivation $A\to A$ whose values at $y$,~$h$ and~$x$ are
$u_1$,~$u_2$ and~$u_3$, respectively, \emph{assuming there is
one}.

\begin{Lemma}
The algebra~$A$ has no non-zero locally nilpotent homogeneous
derivations.
\end{Lemma}

\begin{proof}
Let $d:A\to A$ be a locally nilpotent homogeneous derivation. Suppose
first that $D=k[h^{\pm 1}]$. As we observed above, $d$ vanishes on $h$
and $h^{-1}$ because they are units, so $d(yx)=d(a)=0$ and therefore
$d(y)=d(x)=0$: we see that $d=0$.

Let now $D$ be $k[h]$ and $r$ be the weight of~$d$. We will assume
that $r>0$; if we had $r<0$ the same reasoning would apply, and the
situation is even simpler if $r=0$. There are homogeneous elements of
positive weight in $\ker d$, so there exist $s\in\NN$ and $u\in k[h]$
such that $d(y^su)=0$. Since $\ker d$ is factorially closed, this
implies that in fact $d(y)=0$. On the other hand, there is a
polynomial $p\in k[h]$ such that $d(h)=y^rp$, and from the relation
$hy=qyh$ we see that $y^r p y = q y^{r+1}p$, so that $\sigma(p)=qp$:
it~follows from this that we can write $p=p_1h$ for some $p_1\in
k[h]$. If~$k\geq0$, then $d(A_kh)\subseteq A_{k+r}h$: indeed, if $f\in
k[h]$ we have 
  \[
  d(y^kfh) = y^kd(f)h+y^kfd(h) = y^kd(f)h + y^kfy^rp_1h \in A_{k+r}h
  \]
because $d(f)\in A_r$. This tells us that $d^i(h)\in A_{ir}h$ for all
$i\geq0$. If $i_0=\deg_d(h)$, then $0\neq d^{i_0}(h)\in
A_{i_0r}h\cap\ker d$ and, since $\ker d$ is factorially closed,
$d(h)=0$. An immediate consequence of this is that
$yd(x)=d(yx)=d(a)=0$, so also $d(x)=0$, and we see that $d=0$, as we
wanted.
\end{proof}

That $D$ is $k[h]$ or $k[h^{\pm1}]$ is important in this lemma, for
these two algebras have very few locally nilpotent derivations. Let us
exhibit an example with $D=k[h_1,h_2]$ where its conclusion does not
hold. We take $q\in k\setminus\{0,1\}$ a root of unity of order~$e>1$,
$\sigma:D\to D$ the automorphism such that $\sigma(h_i)=qh_i$ for
$i\in\{1,2\}$, an arbitrary $a\in D$ and consider the algebra
$A=\A(k[h_1,h_2],\sigma,a)$. There is a unique derivation $\tilde
d:D\to D$ such that $\tilde d(h_1)=h_2$ and $\tilde d(h_2)=0$, and it
is locally nilpotent, and using this it is easy to check that for each
$r>0$ there is a locally nilpotent derivation $d_r:A\to A$ with
$d_r(y)=0$, $d_r(h_1)=y^{re}h_2$, $d_r(h_2)=0$ and
$d_r(x)=y^{re-1}\tilde d(a)$. Since $d_r$ is clearly homogeneous, we
see that the conclusion of the lemma does not apply to~$A$.

\begin{Corollary}\label{corollary:lfd}
The locally finite derivations of~$A$ are homogeneous of weight zero.
\end{Corollary}

\begin{proof}
Let $d:A\to A$ be a locally finite derivation. Since $A$ is finitely
generated, there are non-zero homogeneous derivations
$d_1$,~\dots,~$d_l:A\to A$ of strictly increasing weights such that
$d=d_1+\cdots+d_l$. The weight of~$d_l$ cannot be positive, for then
$d_l$~would be locally nilpotent ---because~$d$ is locally finite---
and the lemma would imply that $d_l=0$; similarly, the weight of~$d_1$
cannot be negative. It follows that $d$ itself is homogeneous of
weight zero.
\end{proof}

\begin{Proposition}\label{prop:lfd}
Let $d:A\to A$ be a locally finite derivation, and consider the
derivation $\xi = y\der{y}-x\der{x}$.
\begin{thmlist}

\item If $a$ is not a monomial then $d$ is a scalar multiple of $\xi$.

\item If $a$ is a monomial then $d$ is a linear combination of $\xi$
and $\tau = h\der{h}+Nx\der{x}$.

\end{thmlist}
All locally finite derivations are diagonalizable with the standard
monomials as eigenvectors and, in particular, they commute.
\end{Proposition}

We will refer to $\xi:A\to A$ in what follows as the \emph{Eulerian
derivation} of~$A$. It is easy to check that its eigenvalues are
exactly the integers, and that for each $r\in\ZZ$ the eigenspace
of~$\xi$ corresponding to~$r$ is precisely~$A^{\wt r}$, the
homogeneous component of~$A$ of weight~$r$.

\begin{proof}
According to Corollary~\ref{corollary:lfd} the derivation~$d$ is of
weight zero, so there are polynomials $p_1$,~$p_2$,~$p_3\in D$ such
that $d=yp_1\der{y}+p_2\der{h}+p_3x\der{x}$. In particular
$d$~restricts to a locally finite derivation $D\to D$, and therefore
this restriction has to be of the form $(\alpha h + \beta)\der h$,
with $\alpha$, $\beta\in k$. Looking at the coefficients of~$y$ in
both sides of the equality $d(hy)=qd(yh)$, we see that in fact
$\beta=0$.

There is a sequence $(g_i)_{i\geq0}$ in~$D$ such that $g_0=1$,
$d^i(y)=yg_i$ and $g_{i+1}=p_1g_i+\alpha g_i'h $ for all $i\geq0$. If
$D=k[h]$ we have $\deg g_i=i\deg p_1$ and the local finiteness of~$d$
implies that $p_1\in k$; if $D=k[h^{\pm 1}]$ we reach the same
conclusion by considering the degree of the first or last monomials of
the~$g_i$.

Applying $d$ to both sides of the equality $yx=a$, we see that
$a\sigma^{-1}(p_1+p_3)=\alpha a'h$, which is possible only if $p_3\in
k$. If we now solve this equation for the three scalars $p_1$,
$\alpha$ and $p_3$ we obtain the claims~\thmitem{1} and~\thmitem{2} of
the statement. The last claim, finally, can be proved directly by
inspection.
\end{proof}

Since the dimension of the vector space of locally finite derivations
of an algebra is invariant under isomorphisms, the above
Proposition~\ref{prop:lfd} has the following consequence:

\begin{Corollary}\label{coro:mon}
If $A_1=\A(D,q_1,a_1)$ and $A_2=\A(D,q_2,a_2)$ are two isomorphic
quantum generalized Weyl algebras, then either both $a_1$ and~$a_2$
are monomials or neither of them are. Moreover if one of them is a
unit the other one also.
\end{Corollary}

\begin{proof}
The first claim is an immediate consequence of the proposition. On the
other hand, it is easy to see that $a_1$ is a unit if and only if
$A_1^\times/k^\times$ is a non-trivial group; in that case, it is
isomorphic to $\ZZ$ when $D=k[h]$ and to $\ZZ^2$ when $D=k[h^{\pm
1}]$. As this quotient is invariant under isomorphisms of
$k$-algebras, the second claim follows.
\end{proof}

We are now in position to establish  the key fact that will allow us
to describe the isomorphisms and automorphisms of our algebras in the
next section:

\begin{Proposition}\label{prop:euler}
Let $A_1=\A(D,q_1,a_1)$ and $A_2=\A(D,q_2,a_2)$ two quantum
generalized Weyl algebras with $a_1$ and $a_2$ not units, and let
$\xi_1$ and $\xi_2$ be their respective Eulerian derivations. If
$\eta:A_1\to A_2$ is an isomorphism, then
$\eta\circ\xi_1\circ\eta^{-1}$ is a scalar multiple of~$\xi_2$.
\end{Proposition}

\begin{proof}
Let us write $\xi_2'=\eta\circ\xi_1\circ\eta^{-1}$, which is a locally
finite derivation of~$A_2$. If~$a_2$ is not a monomial, then the first
part of Proposition~\ref{prop:lfd} immediately implies that $\xi_2'$
must be a scalar multiple of~$\xi_2$. We need only consider, then, the
case where $a_2=h^{N_2}$ is a monomial and therefore, by our
assumption that $a_2$ is not a unit, that $N_2>0$ and $D=k[h]$. We
have, then, a derivation $\tau_2'=\eta\circ\tau_1\circ\eta^{-1}$ with
the notation of Proposition~\ref{prop:lfd}. The second part of that
proposition implies that there is a matrix
$M=\psm{m_{1,1}&m_{1,2}\\m_{2,1}&m_{2,2}}\in\GL_2(k)$ such that
  \begin{equation} \label{eq:M}
  \begin{pmatrix}\xi_2'\\\tau_2'\end{pmatrix}
  =
  M \begin{pmatrix}\xi_2\\\tau_2\end{pmatrix}.
  \end{equation}

If $i\in\{1,2\}$, the derivations $\xi_i$ and $\tau_i$ are
simultaneously diagonalizable with integer eigenvalues, so there is a
direct sum decomposition $A_i=\bigoplus_{\lambda\in\ZZ^2}A_i^\lambda$
with $A_i^\lambda=\{u\in
A_i:\xi_i(u)=\lambda_1u,\tau_i(u)=\lambda_2u\}$ for all
$\lambda=(\lambda_1,\lambda_2)\in\ZZ^2$, which is a $\ZZ^2$-grading.
The set $\Lambda_i=\{\lambda\in\ZZ^2:A_i^\lambda\neq0\}$ is a
submonoid of~$\ZZ^2$, and it is generated as such by $(1,0)$, $(0,1)$
and $(-1,N_i)$ because $y$, $h$ and $x$ are, respectively, of those
degrees. Morover, the vectors $(1,0)$ and $(-1,N_i)$ are the unique
indecomposable elements of~$\Lambda_i$ which are not interior to the
convex hull of~$\Lambda_i$; see Figure~\vref{figura}.

  \begin{figure}
  \begin{tikzpicture}[scale=0.8]
  \fill[black!15] (0,0) -- (3.25,0) -- (3.25,4.25) 
                        -- ($(-4.25/2,4.25)$) -- cycle ;
  \draw[thin, ->] (-2.5,0) -- (3.5,0);
  \draw[thin, ->] (0,-.5) -- (0,4.5);
  \draw[very thick] (0,0) -- ($(-4.5/2,4.5)$);
  \draw[very thick] (0,0) -- (3.5,0);
  \foreach \v in { (-1,2) , (1,0), (0,1) }
    \draw[line width=3pt, red, -latex] (0,0) -- \v;
  \foreach \x in { -2, ..., 3 }
    \foreach \y in { 0, ..., 4 }  
      \pgfmathparse{\y + 2 * \x >= 0 ? "fill" : ""}
      \draw[\pgfmathresult] (\x,\y) circle(2pt) ;
  \node[anchor=center] at (1.5,3.5) {$\Lambda_i$};
  \end{tikzpicture}
  \caption{The semigroup $\Lambda_i$.}\label{figura}
  \end{figure}
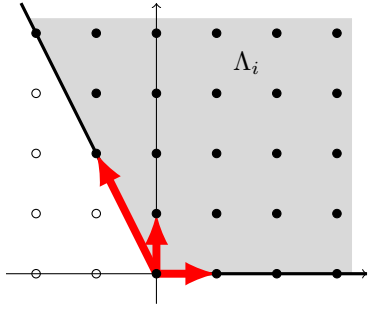

If $\lambda\in\ZZ^2$ and $u\in A_2^\lambda$, we see from~\eqref{eq:M}
that $\eta^{-1}(u)\in A_1^{M\lambda}$: it follows that $\eta^{-1}$
restricts to an isomorphism $A_2^{\lambda}\to A_1^{M\lambda}$ and,
therefore,  the linear map $\lambda\in\ZZ^2\mapsto M\lambda\in\ZZ^2$
induced by~$M$ restricts to an isomorphism
$\phi:\Lambda_2\to\Lambda_1$. As $(1,0)$ and $(0,1)$ are in
$\Lambda_2$ and their images under this map are in~$\ZZ^2$, we see
that $M$ has integral coefficients; considering the inverse
map~$\eta^{-1}$ we see that the same applies to $M^{-1}$, so that
$M\in\GL_2(\ZZ)$. Since the restriction~$\phi$ must preserve
indecomposable non-interior points, it maps the vectors $(1,0)$ and
$(-1,N_2)$ to $(1,0)$ and $(-1,N_1)$ in some order. Exploring the two
possibilities shows that $N_1=N_2$ and that $M$ is either the identity
matrix or $\begin{psmallmatrix}-1&0\\N_1&1\end{psmallmatrix}$. In any
case, we see that $\xi_2'\doteq\xi_2$.
\end{proof}

We remark that the result of this proposition is false when $a_1$ (and
then $a_2$) is a unit. For example, there is an automorphism
$\eta:A(D,q,1)\to A(D,q,1)$ such that $\eta(x)=x$, $\eta(h)=hx$,
$\eta(y)=y$, and it does not preserve the Eulerian derivation. This
fact is what forces us to consider this case separately.

\section{Automorphisms and isomorphisms}
\label{sect:proofs}

We have shown that isomorphisms of quantum generalized Weyl algebras
preserve, up to scalars, their Eulerian derivations. This fact evinces
a non-trivial rigidity of these algebras which strongly restricts the
form of isomorphisms between them:

\begin{Proposition}\label{prop:iso-h}
If $\eta:A_1(D,q_1,a_1)\to A_2(D,q_2,a_2)$ is an isomorphism of
quantum generalized Weyl algebras with $a_1$ and $a_2$ not units, then
there exist $\gamma \in k$, $\epsilon \in \{\pm 1\}$
and~$\mu$,~$\nu\in D^\times$ such that $\eta(h)=\gamma h^\epsilon$ and 
\begin{itemize}

\item[\circitem{1}] either $\eta(y)=y\mu$ and $\eta(x)=\nu x$ 

\item[\circitem{2}] or $\eta(y)=\mu x$ and $\eta(x)=y\nu$.
\end{itemize}
If $D=k[h]$ then necessarily $\epsilon=1$ and $\mu$,~$\nu\in k^\times$.
\end{Proposition}

\begin{proof}
According to Proposition~\ref{prop:euler}, there exists a non-zero
scalar $\lambda\in k$ such that   
  \begin{equation}\label{eq:etaxi}
  \eta\circ\xi_1=\lambda\,\xi_2\circ\eta
  \end{equation}
and, therefore, for each $r\in\ZZ$ the subspace $\eta(A_1^{\wt r})$ is
the eigenspace of~$\xi_2$ corresponding to the eigenvalue~$r/\lambda$;
in particular, $D=\ker\xi_2=\eta(\ker\xi_1)=\eta(D)$ and
$\eta$~restricts to an algebra isomorphism~$D\to D$. As $\xi_2$ has
integer eigenvalues, we must have $\lambda\in\{\pm1\}$. 

Let us suppose that~$\lambda=1$; the other possibility can be handled
similarly and will lead to the second possibility~\circitem{2} in the
statement. There exists an $f\in D$ such that
$y=\eta(yf)=\eta(y)\eta(f)$: since $\eta(f)\in D$, this implies that
$\eta(y)$~generates $A_2^\wt1$ as a right $D$-module. This module is
free of rank one and $y$ and~$\eta(y)$ are two generators: it~follows
that there is an unit~$\mu\in D$ such that $\eta(y)=y\mu$. The same
argument applied to~$x$ shows that there is also an unit $\nu\in D$
such that $\eta(x)=\nu x$.

Consider the case $D=k[h]$. As the restriction $\eta:D\to D$ is an
isomorphism, we have $\eta(h)=\gamma h+\delta$ for some~$\gamma\in
k\setminus0$ and $\delta\in k$. Since $hy=q_1yh$ in~$A_1$, we have
$(\gamma h+\delta)\mu y=q_1 \mu y(\gamma h+\delta)$ in~$A_2$. We
conclude that $\delta=0$ and the proposition follows. In the case
$D=k[h^{\pm 1}]$ we must have $\eta(h)=\gamma h^\epsilon$ for some
$\epsilon\in\{\pm 1\}$ since $h$ generates $D^\times / k^\times\cong
\ZZ$.
\end{proof}

From the locally nilpotent derivation $d$ of the algebra
$A=\A(k[h_1,h_2],\sigma,a)$ considered in the example given in
Section~\ref{sect:ders}, we obtain, by exponentiation, a
$1$\nobreakdash-parameter family of automorphisms $\eta_t:A\to A$ such
that   
  \begin{align*}
  &\eta_t(y) = y,  
        &&\eta_t(h_1) = h_1 + t y^{re}h_2, 
        \\
  &\eta_t(x) = y^{-1}a(h_1+ty^{re}h_2,h_2) 
        &&\eta_t(h_2)=h_2,
  \end{align*}
which is neither homogeneous not linear. This shows that the
conclusion of the proposition above does not hold when $D=k[h_1,h_2]$
and, in fact, as this construction can be carried out starting from
any locally nilpotent derivation of $D$ ---in this case the group of
automorphisms is much larger.

\medskip

At this point, we have everything we need to prove the theorems from
the introduction.

\begin{proof}[Proof of Theorem~\ref{thm:isos}]
The sufficiency of the condition can be checked by a straightforward
verification, which we omit, so we only prove the necessity. 

Let $\eta:A_1\to A_2$ be an isomorphism. From
Proposition~\ref{prop:iso-h} we know there is $\gamma \in k$,
$\epsilon \in \{\pm 1\}$ and~$\mu$,~$\nu\in D^\times$ such that
$\eta(h)=\gamma h^\epsilon$ and \circitem{1} either $\eta(y)=y\mu$ and
$\eta(x)=\nu x$ \circitem{2} or $\eta(y)=\mu x$ and $\eta(x)=y\nu$; if
$D=k[h]$ then moreover $\epsilon=1$ and $\mu$,~$\nu\in k^\times$. If
we are in the first case, we have
  \begin{gather*}
  \sigma^{-1}(\mu\nu)a_2(h)
    = y\mu\nu x
	= \eta(yx)
	= \eta(a_1(h))
	= a_1(\gamma h^\epsilon)
	\\
\shortintertext{and}
  \gamma q_2^\epsilon yh^\epsilon\mu
    = \gamma h^\epsilon y\mu
	= \eta(hy)
	= q_1\eta(yh)
	= \gamma q_1y\mu h^\epsilon.
  \end{gather*}
As $\sigma^{-1}(\mu\nu)\in D^\times$, the necessity of the conditions
is clear. 
\end{proof}

\begin{proof}[Proof of Theorem~\ref{thm:autos-1}]
The verification that the set $G$ is indeed a subgroup of~$\Aut(A)$ is
routine, so we only check~\thmitem{1} and~\thmitem{2}. Let $\eta:A\to
A$ be an automorphism. According to Proposition~\ref{prop:iso-h},
there are $\gamma$,~$\mu$,~$\nu\in D^\times=k\setminus0$ such that
$\eta(h)=\gamma h$ and either \circitem{1}~$\eta(y)=y\mu $ and
$\eta(x)=\nu x$, or \circitem{2}~$\eta(y)=\mu x$ and $\eta(x)=y\nu$.
If \circitem{1}~holds, applying~$\eta$ to both sides of the
equality~$yx=a(h)$ shows that 
  \begin{equation}\label{eq:munu}
  a_i\neq0 \implies \gamma^i=\mu\nu, 
  \end{equation}
so that $\gamma^{i-j}=1$ whenever $a_ia_j\neq0$ and, in consequence,
$\gamma\in C_g$. Additionally,  \eqref{eq:munu}~tells us that
$\nu=\mu^{-1}\gamma^N$ and then we see that $\eta=\eta_{\gamma,\mu}\in
G$.

If instead \circitem{2}~holds, applying~$\eta$ to the
equality~$hy=qyh$ shows that $q^2=1$ so that in fact $q=-1$. This
means that when $q=-1$ the alternative~\circitem{2} does not occur,
and~$\Aut(A)=G$. On the other hand, if~$q=-1$ there is indeed an
automorphism~$\Omega$ as described in the statement, and
$\eta\circ\Omega\in G$ because this composition falls in the
case~\circitem{1} with which we have already dealt. The subgroup~$G$
together with~$\Omega$ thus generate~$\Aut(A)$ in this situation and
all the claims in~\thmitem{2} now follow at once.
\end{proof}

We need two lemmas for the proof of Theorem~\ref{thm:autos-2}; the
notation is as in the statement of that theorem.

\begin{Lemma}
The Laurent polynomial $a$ is symmetric if and only if there exists an
automorphism $\Omega_\sym:A\to A$ such that $\Omega_\sym(h)\doteq
h^{-1}$.
\end{Lemma}

\begin{proof}
If $a$ is symmetric then there exist $l\in\NN$ and $\gamma$, $\delta
\in k$ such that $\delta a(h)=h^l  a(\gamma h^{-1})$. The automorphism
$\Omega_\sym$ is defined by
  \begin{align*}
    &\Omega_\sym(y)= x,
    &&\Omega_\sym(h)=q^{-1}\gamma h^{-1},
    && \Omega_\sym(x)= \delta q^{-l}y h^{-l}.
  \end{align*}
Conversely, if exists such an automorphism $\Omega_\sym$ then

  \[
  \Omega_\sym(y)\Omega_\sym(x)
        = \Omega_\sym(yx)
        = \Omega_\sym(a)
        = a(\gamma h^{-1}),
  \]
and it is easy to see, applying Proposition~\ref{prop:iso-h}, that the
left hand side of this equation is equal to $\delta a(h) h^{-l}$ for
some $\delta\in k$ and $l\in\ZZ$.
\end{proof}

\begin{Lemma}
The parameter $q$ is equal to $-1$ if and only if there exists an
automorphism $\Omega_{-1}:A\to A$ such that $\Omega_{-1}(y)\doteq x$,
$\Omega_{-1}(x)\doteq y$ and $\Omega_{-1}(h)\doteq h$.
\end{Lemma}

\begin{proof}
If $q=-1$ then $\Omega_{-1}$ is defined by
 \begin{align*}
 &\Omega_{-1}(y) =x, 
        && \Omega_{-1}(h)=qh, 
        && \Omega_{-1}(x)=y.
 \end{align*}
If there exists an automorphism as in the statement then
 \[
  hx \doteq \Omega_{-1}(h)\Omega_{-1}(y)
        = \Omega_{-1}(hy)
        = q\Omega_{-1}(yh)
        = q\Omega_{-1}(y)\Omega_{-1}(h)
        \doteq qxh,
 \]
so that $q=q^{-1}$.
\end{proof}

\begin{proof}[Proof of Theorem~\ref{thm:autos-2}]
Let $\eta\in \Aut(A)$. From Proposition~\ref{prop:iso-h}, we know that
$\eta(h)\doteq h^{\pm 1}$. If $\eta(h)\doteq h^{-1}$ then the first
lemma above shows that $a$ is symmetric and, moreover, that
$\eta\circ\Omega_\sym\in K$. If $a$ is not symmetric then we must have
$\eta(h)\doteq h$ and therefore $\eta\in K$. This proves
part~\thmitem{1} of the theorem.

Assume now that $\eta\in K$. In this case, if $\eta(y)\doteq x$ we must
have $q=-1$ and in this case $\eta\circ \Omega_{-1}\in G$. Conversely,
if $q\neq -1$ then necessarily $\eta(y)\doteq y$, so that $K$ is as
in~\thmitem{2}.
\end{proof}

The theorems stated in the introduction leave untouched the case in
which the parameter~$a$ of the generalized Weyl algebras $\A(D,a,q)$
is a unit in~$D$. As promised there, we now state and prove the
corresponding results for this case.

\begin{Theorem}
\begin{thmlist}

\item Let $A=\A(k[h],a,q)$ with $a\in k^\times$. If $q\neq-1$, let $H$
be the subgroup $\{\psm{1&z\\0&1}:z\in\ZZ\}$ and let $H$ be
$\{\psm{\pm1&z\\0&1}:z\in\ZZ\}$ otherwise. There is then a right-split
short exact sequence of groups
  \[
  \xymatrix{
  0 \ar[r]
        & (k^\times)^2 \ar[r]
        & \Aut(A) \ar[r]
        & H \ar[r]
        & 1
  }\]

\item Let $A=\A(k[h^{\pm1}],a,q)$ with $a=\alpha h^N$ and $N\in\ZZ$.
If $q\neq-1$, let $H$ be $\SL_2(\ZZ)$ if $q\neq-1$ and $\GL_2(\ZZ)$
otherwise.  There is then a right-split short exact sequence of groups
  \[
  \xymatrix{
  0 \ar[r]
        & (k^\times)^2 \ar[r]
        & \Aut(A) \ar[r]
        & H \ar[r]
        & 1
  }\]

\end{thmlist}
\end{Theorem}

\begin{proof}
In both cases the algebra $A$ is generated by $x$ and~$h$, because $a$
is a unit.

\thmitem{1} We proceed exactly as in the beginning of the proof of
Proposition~\ref{prop:euler}. Given an automorphism $\eta:A\to A$,
this constructs a matrix $M\in\GL_2(\ZZ)$ such that for all
$\lambda\in\ZZ^2$ and $u\in A^\lambda$, we have that $\eta^{-1}(u)\in
A^{M\lambda}$ and which therefore preserves the subsemigroup
$\Lambda\subseteq\ZZ^2$ which, in this case, is generated by
$(\pm1,0)$ and $(0,1)$; we remark that in this situation the semigroup
$\Lambda$ does not have indecomposable elements, so that the argument
given in the proof of Proposition~\ref{prop:euler} cannot be
continued. In any case, as $M$ preserves~$\Lambda$, we must have
$M=\begin{psmallmatrix}\epsilon&\ell\\0&1\end{psmallmatrix}$ for some
$\epsilon\in\{\pm1\}$ and $\ell\in\ZZ$. Since $A^\lambda$ is
one-dimensional for all $\lambda\in\Lambda$, this implies that
  \begin{align}
  &\eta^{-1}(x)\doteq x^\epsilon, &&\eta^{-1}(h)\doteq hx^\ell. \label{eq:eta}
  \end{align}
From the relation $xh=qhx$ we see that if $q\neq-1$ we must have
$\epsilon=1$, so that $M\in H$. In this way we obtain a morphism of
groups $\pi:\Aut(A)\to H$, and it is easy to see that it is surjective
and right-split ---one can use formulas~\eqref{eq:eta} to construct a
section. The kernel of~$\pi$, isomorphic to~$(k^\times)^2$, can be
identified at once.

\thmitem{2} There an obvious group homomorphism
$\pi:\Aut(A)\to\Aut(A^\times/k^\times)$. Since $A$ is generated by $x$
and $h$, which are units, the kernel of~$\pi$ is easily seen to be the
group of automorphisms which multiply those generators by non-zero
scalars and therefore isomorphic to~$(k^\times)^2$. 

It is easy to see that $A^\times/k^\times$ is an abelian group freely
generated by the classes of~$h$ and~$x$, so we can identify it
with~$\ZZ^2$. If $\eta:A\to A$ is an automorphism and
$M=\pi(\eta)=\psm{m_{1,1}&m_{1,2}\\m_{2,1}&m_{2,2}}$, we must have
  \begin{align}
  &\eta(h) \doteq h^{m_{1,1}}x^{m_{2,1}},
  &\eta(x) \doteq h^{m_{1,2}}x^{m_{2,2}}. \label{eq:eta-2}
  \end{align}
From the $q$-commutation relation between $h$~and~$x$, we see that
necessarily $\det M=1$ if~$q\neq1$; this means that, in any case, $M$
is in the subgroup~$H$ and $\pi$ can be corestricted to a morphism
$\Aut(A)\to H$. Using formulas~\eqref{eq:eta-2} we can easily
construct a section for this map, thereby finishing the proof of the
theorem.
\end{proof}

An argument completely parallel to that of this proof
establishes the following final result. We omit the details.

\begin{Theorem}
Let $D$ be $k[h]$ or $k[h^{\pm1}]$, let $a_1$,~$a_2\in D$ be two
units, and let $q_1$,~$q_2\in k\setminus\{0,1\}$. The algebras
$\A(D,a_1,q_1)$ and $\A(D,a_2,q_2)$ are isomorphic
iff~$q_2\in\{q_1,q_1^{-1}\}$.~\qed
\end{Theorem}

\end{document}